\documentclass{amsart}
\usepackage[left=3.5cm, right=2.5cm, top=3cm, bottom=3cm]{geometry}

\usepackage{graphicx}

\usepackage{mg_macros}

\begin{document}


\title[Conditioning of Gaussian processes]{Conditioning of Gaussian processes and \\ a zero area Brownian bridge}

\author{Maik G\"orgens}
\address{Department of Mathematics, Uppsala University}
\curraddr{P.O. Box 480, 751 06 Uppsala, Sweden}
\email{gorgens.maik@gmail.com}

\date{\today}

\begin{abstract}
  We generalize the notion of Gaussian bridges by conditioning Gaussian processes given that certain linear functionals of the sample paths vanish. We show the equivalence of the laws of the unconditioned and the conditioned process and by an application of Girsanov's theorem, we show that the conditioned process follows a stochastic differential equation~(SDE) whenever the unconditioned process does. In the Markovian case, we are able to determine the coefficients in the SDE of the conditioned process explicitly. Our main example is Brownian motion on~$[0,1]$ pinned down in~$0$ at time~$1$ and conditioned to have vanishing area spanned by the sample paths.
\end{abstract}

\keywords{Gaussian processes, Conditioning, Brownian bridge, Series expansions}

\maketitle


\section{Introduction}

Let $X = (X_s)_{s \in [0,T]}$ be a Gaussian process with values in the space of continuous functions $C([0,T])$ and assume $\E X_s = 0$ for all $s \in [0, T]$. Let $A$ be a finite set of linear functionals on $C([0,T])$. In this work we consider the conditioned process of $X$ given that the linear functionals in $A$ acting on $X$ vanish. 
The conditioned process is denoted by $X^{(A)} = (X^{(A)}_s)_{s \in [0,T]}$. A formal definition is given in Section~\ref{S:FORMAL_DEF}.



A well studied example is that of Gaussian bridges (see for example \cite{Gas04} and~\cite{Bau07}). In this case the set $A$ only consists of the element $\delta_T$, where $\delta_T$ denotes point evaluation of a function at point $T$. For example the standard linear Brownian motion on $[0,1]$ conditioned to have $W_1 = 0$ (i.e., $A = \{ \delta_1 \}$) yields the Brownian bridge $B$ on $[0,1]$. An anticipative representation of $B$ is
\equ[E:INTRO_GB_1]{ B_s = W_s - s W_1, \quad 0 \leq s \leq 1, }
and a non-anticipative representation (i.e., adapted to the natural filtration of $W$) of $B$ is
\equ[E:INTRO_GB_2]{ dB_s = dW_s - \frac{B_s}{1-s} \, ds, \quad B_0 = 0, \quad 0 \leq s < 1. }
The present work generalizes the setting of Gaussian bridges by allowing several and more general conditions. Our main example (studied in Section~\ref{SS:ZABB}) is the Brownian motion $W$ conditioned to have $W_1 = 0$ and $I_1 = \int_0^1 W_x \, dx = 0$ (i.e., $A = \{ \delta_1, a_0 \}$ with $a_0(f) = \int_0^1 f(x) \, dx$, $f \in C([0,1])$). We call the conditioned process the \emph{zero area Brownian bridge} and denote it by $M$. Figure~\ref{FIG:1} shows a typical path of $M$.
\begin{figure}
  \includegraphics[width=1\textwidth]{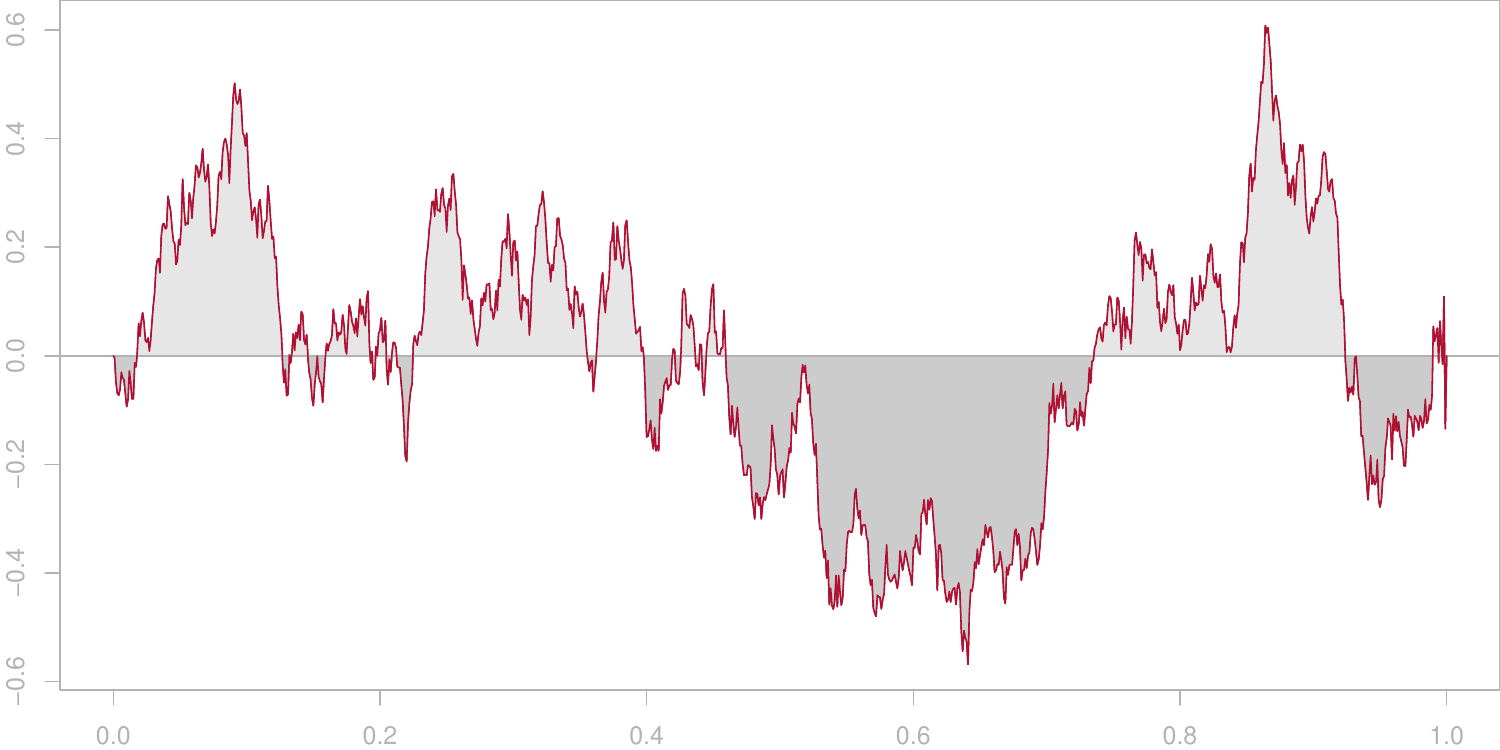}
  \caption{A realization of a zero area Brownian bridge.}
  \label{FIG:1}
\end{figure}
An anticipative representation of $M$ (corresponding to~\eqref{E:INTRO_GB_1} for $B$) is 
\[ M_s = W_s - s(3s - 2) W_1 - 6 s (1-s) I_1, \quad 0 \leq s \leq 1, \]
and a non-anticipative representation of $M$ (corresponding to~\eqref{E:INTRO_GB_2} for $B$) is
\[ dM_s = dW_s - \frac{4M_s}{1-s} ds - \frac{6 J_s}{(1-s)^2} ds, \quad M_0 = 0, \quad 0 \leq s < 1, \]
where $J_s = \int_0^s M_x \, dx$. In particular, the two dimensional process $(M_s, J_s)_{s \in [0,1]}$ is a time-inhomogeneous Markov process. 

On earlier work on conditioned Gaussian processes we mention~\cite{Ali02} and~\cite{Sot14}. In these articles very similar settings as in our work were studied and the resulting processes were called ``generalized Gaussian bridges''. Anticipative as well as non-anticipative representations were given. However, we believe that this paper gives further insights into the theory of conditioned Gaussian processes. In particular, we obtain the non-anticipative representations in a very intuitive way which allows for very explicit calculations. 

We fix some notations and introduce the conditioned process properly. Then we state the main results of the paper.

\subsection{Notations and definition of the conditioned process}\label{S:FORMAL_DEF}

Let $C([0,T])$ be the space of continuous functions on $[0,T]$ equipped with the supremum norm
\[ \|f\|_\infty = \sup_{0 \leq s \leq T} |f(s)|, \qquad f \in C([0,T]). \]
Then $(C([0,T]), \|\cdot\|_\infty)$ becomes a separable Banach space. For a continuous function $f \in C([0,T])$ and an element $a \in C([0,T])^*$ we write $a(f)$ for the evaluation map. Let $\mathcal{C}$ denote the Borel $\sigma$-algebra on $C([0,T])$. The dual space $C([0,T])^*$ of $C([0,T])$ can be identified with the space of signed finite Borel measure on $[0,T]$ (see Appendix~C in~\cite{Con85}). We use the notation $a(f)$ and $\int f(s) \, a(ds)$ interchangeably. In particular, we use the second form if the integration only runs over a subset of $[0,T]$. By $\delta_s$, $s \in [0,T]$, we denote the point evaluation at point $s$, i.e., $\delta_s(f) = f(s)$, $f \in C([0,T])$. For $0 \leq s \leq T$, let $\F_s \subset \mathcal{C}$ be the smallest $\sigma$-algebra on $C([0,T])$ such that all $\delta_r$, $0 \leq r \leq s$, are $\F_s$-$\B(\R)$-measurable, where $\B(\R)$ is the Borel $\sigma$-algebra on $\R$. Note that $\F_T = \mathcal{C}$.

Let $X = (X_s)_{s \in [0,T]}$ be a continuous Gaussian process defined on a probability space $(\Omega, \A, \Prob)$. Assume $\E X_s = 0$ for all $s \in [0,T]$ and let $R_X: [0,T] \times [0,T] \rightarrow \R$ be the covariance function of $X$, $R_X(s,t) = \E X_s X_t$. A \emph{condition} for $X$ is an element $a \in C([0,T])^*$ and $X$ \emph{fulfills the condition} $a$ if $a(X) = 0$, almost surely. Let $A \subset C([0,T])^*$ be a finite set of conditions. We define a probability measure $\Prob_X^{(A)}$ on $(C([0,T]), \mathcal{C})$ by
\equ[E:INTRO_2]{ \Prob_X^{(A)}(F) = \Prob_X\left(F \ \Big| \bigcap_{a \in A} a^{-1}(0) \right), \quad F \in \mathcal{C}, }
where $\Prob_X$ is the induced measure of $X$ on $(C([0,T]), \mathcal{C})$.

Note that the fact that we condition by an event of probability zero in~\eqref{E:INTRO_2} does not raise a problem: define the set $\mathcal{C}' \subset \mathcal{C}$ by
\[ \mathcal{C}' = \bigcup_{n \in \N } \quad \bigcup_{b_1, \ldots, b_n \in C([0,T])^*} \sigma\{ b_1, \ldots, b_n \}, \]
where $\sigma\{ b_1, \ldots, b_n \} \subset \mathcal{C}$ is the smallest $\sigma$-algebra which makes the functionals $b_1, \ldots, b_n \in C([0,T])^*$ measurable. If we consider $\Prob_X^{(A)}$ only on $\mathcal{C}'$, then $\Prob_X^{(A)}$ is well-defined since conditioning on the event that the Gaussian random variables $a(X)$ vanish for all $a \in A$ becomes orthogonal projection in $\R^n$ (see also Section~9.3 in~\cite{Jan97}). The set $\mathcal{C}'$ is a ring and $\Prob_X^{(A)}$ a pre-measure on $\mathcal{C}'$. Hence, by Carath\'eodory's extension theorem (see Theorem~1.53 in~\cite{Kle14}), $\Prob_X^{(A)}$ extends in a unique way to a probability measure on $\mathcal{C}$ -- the $\sigma$-algebra generated by $\mathcal{C}'$. (The existence and uniqueness of the extension of $\Prob_X^{(A)}$ from $\mathcal{C}'$ to $\mathcal{C}$ also follows from Theorem~\ref{T:COND_PROC}).


A continuous Gaussian process $X^{(A)} = (X^{(A)}_s)_{s \in [0,T]}$ defined on a probability space $(\Omega', \A', \Prob')$ is a \emph{conditioned process of $X$ with respect to the set of conditions $A$} if its induced measure $\Prob_{X^{(A)}}$ on $(C([0,T]), \mathcal{C})$ coincides with $\Prob_X^{(A)}$. The conditioned process is thus only defined in law.

%

\subsection{Main results}\label{SS:MAIN_RES}

Let $N$ be the number of conditions in $A$. In Section~\ref{S:SERIES} we will introduce a separable Hilbert space $H$ and a linear and bounded operator $u: H \rightarrow C([0,T])$ such that
\equ[E:INTRO_3]{ X = \sum_{i=1}^N \omega_i (u e_i) + \sum_{j=1}^\infty \omega_j' (u f_j) \quad \text{and} \quad X^{(A)} = \sum_{j=1}^\infty \omega_j' (u f_j) }
in law for sequences $(e_i)_{i=1}^N \subset H$ and $(f_j)_{j=1}^\infty \subset H$ such that $\{ e_1, \ldots, e_N, f_1, f_2, \ldots \}$ forms an orthonormal basis in $H$, and sequences of independent standard normal random variables $(\omega_i)_{i=1}^N$ and $(\omega'_j)_{j=1}^\infty$. Based on these series expansions we find basic properties of the conditioned process. In particular its covariance structure (Proposition~\ref{P:COND_COVARIANCE}) and an anticipative representation (Theorem~\ref{T:ANTI}).

Let $(e_i)_{i=1}^N \subset H$ and $(f_j)_{j=1}^\infty \subset H$ be as in~\eqref{E:INTRO_3} and let $H^{(A)}$ be the closed linear span of $\{ f_j : j \geq 1 \}$. In Section~\ref{S:EQU_MEAS} we show that $\Prob_X$ and $\Prob_{X^{(A)}}$ are equivalent on $\F_s$ if and only if for every $e_i$ there is an $e_i' \in H^{(A)}$ with $(ue_i')(x) = (ue_i)(x)$, for all $0 \leq x \leq s$.

In Section~\ref{S:NON_ANTI} we show that, under some assumptions on $X$ and $A$, the process $X^{(A)}$ solves a stochastic differential equation of the form
\equ[E:INTRO_4]{ dX^{(A)}_s = \alpha dW_s + \delta(s, X^{(A)}) ds, \quad X^{(A)}_0 = 0, \quad 0 \leq s < T, }
where $W$ is a standard linear Brownian motion and $\delta$ is a progressively measurable functional on $C([0,T])$.

In Section~\ref{S:MARKOV} we assume that $X$ is a Markov process. Defining $I^{(A), 1}, \ldots, I^{(A), N}$ by $I^{(A), i}_s = \int_0^s X^{(A)}_x \, a_i(dx)$, where $A = \{ a_1, \ldots, a_N \}$, it is shown in Theorem~\ref{T:MP} that $(X^{(A)}_s, I^{(A), 1}_s, \ldots, I^{(A), N}_s)_{s \in [0,T]}$ is a Markov process as well. Based on this result we find a formula for $\E [X^{(A)}_t \mid \F^{X^{(A)}}_s]$, $s \leq t$, where $\F^{X^{(A)}}$ denotes the natural filtration of $X^{(A)}$ (Theorem~\ref{T:EXP_FUTURE}), which enables us to find the $\delta$ in~\eqref{E:INTRO_4} explicitly.


\section{A series expansion and basic properties of $X^{(A)}$}\label{S:SERIES}

The aim of this section is to find a series expansion of $X^{(A)}$ analogous to that in~\eqref{E:INTRO_3}. As a preliminary we start with a subsection on processes generated by an operator.

\subsection{Gaussian processes generated by an operator}\label{S:OP_GEN_PROC}

Let $v: H \rightarrow C([0,T])$ be a linear and bounded operator from a separable Hilbert space $H$ into $C([0,T])$ and let $v^*: C([0,T])^* \rightarrow H$ be the adjoint operator of $v$, i.e., $\la v^*a, h \ra = a(vh)$ for all $h \in H$ and $a \in C([0,T])^*$. Let $\la \cdot, \cdot \ra$ denote the scalar product on $H$ and $\|\cdot\|$ its induced norm.

For an orthonormal basis $(h_i)_{i=1}^\infty \subset H$ define
\equ[E:CP_3]{ Z_s = \sum_{i=1}^\infty \omega_i (v h_i)(s) = \sum_{i=1}^\infty \omega_i \la v^* \delta_s, h_i \ra, }
where $(\omega_i)_{i=1}^\infty$ is a sequence of independent standard normal random variables. The series on the right hand side of~\eqref{E:CP_3} converges almost surely for each $s \in [0,T]$ because of
\[ \sum_{i=1}^\infty |\la v^* \delta_s, h_i \ra|^2 = \|v^* \delta_s\|^2 < \infty. \]
The exceptional null set in~\eqref{E:CP_3} in general depends on $s \in [0,T]$. So~\eqref{E:CP_3} defines a not necessarily continuous Gaussian process $Z=(Z_s)_{s \in [0,T]}$. If the series
\[ Z = \sum_{i=1}^\infty \omega_i (v h_i) \]
converges almost surely in $C([0,T])$ we say that $v$ \emph{generates} the continuous Gaussian process $Z$ ($v$ is also called \emph{associated operator} of $Z$).

For $a,b \in C([0,T])^*$ we have 
\equ[E:CP_4]{ \E a(Z) b(Z) = \sum_{i=1}^\infty a(v h_i) b(v h_i) = \la v^* a, v^* b \ra. }
In particular, for the covariance function $R_Z(s,t) = \E Z_s Z_t$ of $Z$ it holds $R_Z(s,t) = \la v^* \delta_s, v^* \delta_t \ra$. Hence, a change of the orthonormal basis in~\eqref{E:CP_3} gives another process $Z'$, in general different from $Z$. But, by~\eqref{E:CP_4}, $Z$ and $Z'$ have the same finite-dimensional distributions.

\subsection{A series expansion of the process $X^{(A)}$}

The following result will be crucial for our work.

\begin{theorem}[Theorem~3.5.1 in~\cite{Bog98}]\label{T:GEN_PROCESS}
  For the continuous Gaussian process $X=(X_s)_{s \in [0,T]}$ there is a separable Hilbert space $H$ and a linear and bounded operator $u:H \rightarrow C([0,T])$ such that, for every orthonormal basis $(h_i)_{i=1}^\infty \subset H$,
  \equ[E:CP_2b]{ X = \sum_{i=1}^\infty \omega_i (uh_i) }
  in distribution. In particular, the series on the right hand side converges almost surely in $C([0,T])$.
\end{theorem}

We define the closed linear subspace
\[ H^{(A)} = \{ h \in H : a(uh) = 0 \quad \text{for all $a \in A$} \} \subset H \]
and call it the \emph{reduced Hilbert space} with respect to $A$. Let $H_{(A)} \subset H$ be the orthogonal complement of $H^{(A)}$ (we write $H_{(A)} = H \ominus H^{(A)}$). We call $H_{(A)}$ the \emph{detached subspace} of $H$ with respect to $A$. By definition of $u^*$, 
\al{ H^{(A)} &= \{ h \in H : \la u^*a, h \ra = 0 \quad \text{for all $a \in A$} \} \\
             &= \{ h \in H : \text{$h$ is orthogonal to $u^*a$ for all $a \in A$} \} \subset H, }
and thus $H_{(A)}$ is spanned by the elements $u^*a$,
\equ[E:CONDITIONS]{ H_{(A)} = \setspan \{ u^*a : a \in A \}, }
implying that $H_{(A)}$ is (at most) of dimension $N$.

Define
\equ[E:CP_2]{ X^{(A)} = \sum_{i=1}^\infty \omega_i (u f_i), }
where $(f_i)_{i=1}^\infty \subset H^{(A)}$ is an orthonormal basis in $H^{(A)}$. Applying~\eqref{E:CP_4} for the operator $u$ restricted to $H^{(A)}$, we see that the law of $X^{(A)}$ is independent of the choice of the orthonormal basis in $H^{(A)}$ and since~\eqref{E:CP_2} differs from~\eqref{E:CP_2b} only by a finite number of terms (given that we assume that $\{ f_1, f_2, \ldots \}$ is a subset of $\{ h_1, h_2, \ldots \}$) the series in~\eqref{E:CP_2} converges in $C([0,T])$ almost surely.

\begin{theorem}\label{T:COND_PROC}
  The process $X^{(A)}$ defined in~\eqref{E:CP_2} is a conditioned process of $X$ with respect to $A$.
\end{theorem}

\begin{proof}
  We have to show $\Prob_X^{(A)}(F) = \Prob(X^{(A)} \in F)$ for all $F \in \mathcal{C}$ with $\Prob_X^{(A)}$ defined in~\eqref{E:INTRO_2}. Let $(e_i)_{i=1}^N \subset H_{(A)}$ be an orthonormal basis in $H_{(A)}$. Then the processes $X$ and
  \[ X^{(A)} + \sum_{i=1}^N \omega'_i (ue_i) \]
  coincide in law, where $(\omega'_i)_{i=1}^N$ are independent standard normal distributed random variables independent from $X^{(A)}$. We thus have for $F \in \mathcal{C}$
  \[ \Prob_X^{(A)}(F) = \Prob_X\left(F \ \Big| \bigcap_{a \in A} a^{-1}(0) \right) = \Prob_{X^{(A)} + \sum_{i=1}^N \omega'_i (ue_i)}\left(F \ \Big| \bigcap_{a \in A} a^{-1}(0) \right). \]
  Since $a(X^{(A)}) = 0$ and for all $1 \leq i \leq N$ there is an $a \in A$ such that $a(u e_i) \neq 0$ it follows
  \[ \Prob_X^{(A)}(F) = \Prob_{X^{(A)}}\left(F \ \Big| \bigcap_{a \in A} a^{-1}(0) \right) = \Prob(X^{(A)} \in F ). \qedhere \]

\end{proof}

Let $R_{X^{(A)}}$ be the covariance function of the conditioned process $X^{(A)}$ of $X$ with respect to $A \subset C([0,T])^*$.

\begin{prop}\label{P:COND_COVARIANCE}
  Let $(e_i)_{i=1}^N \subset H_{(A)}$ be an orthonormal basis in the detached subspace $H_{(A)}$. Then
  \[ R_{X^{(A)}}(s,t) = R_X(s, t) - \sum_{i=1}^N (u e_i)(s) (u e_i)(t). \]  
\end{prop}

\begin{proof}
  Let $(f_j)_{j=1}^\infty \subset H^{(A)}$ be an orthonormal basis in $H^{(A)}$. Then an orthonormal basis in $H = H^{(A)} \oplus H_{(A)}$ is $\{ e_1, \ldots, e_N, f_1, f_2, \ldots \}$ and thus, by~\eqref{E:CP_4},
  \[ R_X(s,t) = \sum_{i=1}^N (u e_i)(s) (u e_i)(t) + \sum_{j=1}^\infty (u f_j)(s) (u f_j)(t). \]
  Hence,
  \[ R_{X^{(A)}}(s,t) = \sum_{j=1}^\infty (u f_j)(s) (u f_j)(t) = R_X(s, t) - \sum_{i=1}^N (u e_i)(s) (u e_i)(t). \qedhere \]
\end{proof}

\subsection{Anticipative representation}\label{SS:ANTI}

Define Gaussian processes $I^1, \ldots, I^N$ by
\equ[E:INT_PROC]{ I^i_s = \int_0^s X_x \, a_i(dx), \quad 0 \leq s \leq T, \quad 1 \leq i \leq N. }
In particular, we have $I_T^i = a_i(X)$ for $1 \leq i \leq N$.

\begin{prop}\label{P:ON_COND}
  Given a set of conditions $A = \{ a_1, \ldots, a_N \}$ there is another set of conditions $\hat A = \{ \hat a_1, \ldots, \hat a_M \}$ ($M \leq N$) such that $X^{(A)} = X^{(\hat A)}$ in distribution, the random variables $\hat I_T^1, \ldots, \hat I_T^M$ (defined analogously to~\eqref{E:INT_PROC}) are independent and standard normal, and the set $\{ u^* \hat a_i : 1 \leq i \leq M \}$ is an orthonormal basis in $H_{(\hat A)}$.
\end{prop}

\begin{proof}
  Let the conditions $a_1, \ldots, a_N$ be arbitrary. Then the Gram-Schmidt orthonormalization $\widehat{I}^1_T = I^1_T / \E \left[ I^1_T \right]^2$,
  \equ[E:ANTI_1]{ \widehat{I}^i_T = \mathring{I}^i_T / \E \left[ \mathring{I}^i_T \right]^2, \quad \text{where} \quad \mathring{I}^i_T = I^i_T - \sum_{j=1}^{i-1} \E \left[ I^i_T \widehat{I}^j_T \right] \widehat{I}^j_T, \quad i = 2, \ldots, N, }
  yields independent standard normal random variables $\widehat{I}^1_T, \ldots, \widehat{I}^N_T$ and conditioning on $I^i_T = 0$ is equivalent to conditioning on $\widehat{I}^i_T = 0$ almost surely for $1 \leq i \leq N$ (here we assume without loss of generality that $\mathring{I}^i_T \neq 0$ for all $i = 2, \ldots, N$; if this is not true, we continue only with those $M < N$ many random variables for which it is). Now define measures $\hat{a}_1, \ldots, \hat{a}_N$ by $\hat{a}_1 = a_1 / \E \left[ I^1_T \right]^2$ and
  \[ \hat{a}_i = \mathring{a}_i / \E \left[ \mathring{I}^i_T \right]^2, \quad \text{where} \quad \mathring{a}_i = a_i - \sum_{j=1}^{i-1} \E \left[ I^i_T \widehat{I}^j_T \right] \hat{a}_j. \]
  Then we have $\widehat{I}^i_T = \int_0^T X_x \, \hat{a}_i(dx)$, i.e., we obtain independent standard normal random variables and conditioning with respect to $\{ a_1, \ldots, a_N \}$ is equivalent to conditioning with respect to $\{ \hat{a}_1, \ldots, \hat{a}_N \}$.

  From~\eqref{E:CP_4} it follows for all $1 \leq i, j \leq N$ that
  \[ \la u^* \hat{a}_i, u^* \hat{a}_j \ra = \E \left[ \hat{a}_i(X) \hat{a}_j(X) \right] = \E \left[ \hat{I}^i_T \hat{I}^j_T \right] = \delta_{i,j}, \]
  where $\delta_{i,j}$ denotes the Kronecker delta. Hence, $\{ u^* \hat{a}_1, \ldots, u^* \hat{a}_N \}$ is an orthonormal set in $H$ and thus, by~\eqref{E:CONDITIONS}, an orthonormal basis in $H_{(\hat A)}$.
\end{proof}

The following result follows directly from the general theory of conditioning of Gaussian random variables (see Chapter~9 in~\cite{Jan97}).

\begin{prop}\label{P:ANTI}
  Let $a_1, \ldots, a_N$ be such that the random variables $I_T^1, \ldots, I_T^N$ are independent and standard normal random variables. Then an anticipative representation for $X^{(A)}$ is
  \[ X^{(A)}_s = X_s - \sum_{i=1}^N \E \left[ X_s I_T^i \right] I_T^i. \]
\end{prop}

We drop the requirement that $I^1_T, \ldots, I^N_T$ are orthonormal but we still assume that
the set $\{ u^*a_i : 1 \leq i \leq N \} \subset H_{(A)}$ is linearly independent in $H$. Let $(e_i)_{i=1}^N \subset H_{(A)}$ be an orthonormal basis $H_{(A)}$ and define a matrix $B$ and a vector $b(X)$ by
\[ B = \left( \begin{matrix} a_1(ue_1) & a_1(ue_2) & \dots & a_1(ue_N) \\ a_2(ue_1) & a_2(ue_2) & \dots & a_2(ue_N) \\ \vdots & \vdots & \ddots & \vdots \\ a_N(ue_1) & a_N(ue_2) & \dots & a_N(ue_N) \end{matrix} \right) \quad \text{and} \quad b(X) = \left( \begin{matrix} a_1(X) \\ a_2(X) \\ \vdots \\ a_N(X) \end{matrix} \right). \]

\begin{theorem}\label{T:ANTI}
  The matrix $B$ is invertible and an anticipative representation of the conditioned process $X^{(A)}$ is
  \equ[E:CP_5]{ X^{(A)} = X - \sum_{i=1}^N \xi_i(X) (ue_i), }
  where $\xi(X) = (\xi_1(X), \ldots, \xi_N(X))^\tau$ is given by $\xi(X) = B^{-1} b(X)$.
\end{theorem}

\begin{proof}
  In order to show that the matrix $B$ is invertible, we show that the rank of $B$ is $N$. Since the $e_i$'s form an orthonormal basis in the Hilbert space spanned by $\{ u^*a_1, \ldots, u^*a_N \}$, the rank of $B$ is equal to the rank of $B'$ with
  \[ B' = \left( \begin{matrix} a_1(u u^*a_1) & a_1(u u^*a_2) & \dots & a_1(u u^*a_N) \\ a_2(u u^*a_1) & a_2(u u^*a_2) & \dots & a_2(u u^*a_N) \\ \vdots & \vdots & \ddots & \vdots \\ a_N(u u^*a_1) & a_N(u u^*a_2) & \dots & a_N(u u^*a_N) \end{matrix} \right). \]
  Hence, it is enough to show that the columns of $B'$ are linearly independent. We assume
  \[ \mathbf{0} = \left( \sum_{i=1}^N \lambda_i a_1(uu^*a_i), \ldots, \sum_{i=1}^N \lambda_i a_N(uu^*a_i) \right). \]
  Then,
  \[ 0 = \sum_{j=1}^N \lambda_j \sum_{i=1}^N \lambda_i a_j(uu^*a_i) = \sum_{i,j=1}^N \lambda_i \lambda_j \la u^*a_i , u^*a_j \ra = \left\| \sum_{i=1}^N \lambda_i u^*a_i \right\|^2 \]
  which yields the requirement $\sum_{i=1}^N \lambda_i u^*a_i = 0$ and thus $\lambda_i = 0$, $1 \leq i \leq N$, since $\{ u^*a_1, \ldots, u^*a_N \}$ is assumed to be linearly independent in $H$. Hence, the rank of $B'$ and $B$ is $N$ and the matrix $B$ is invertible.

  Formula~\eqref{E:CP_5} follows from
  \[ X = X^{(A)} + \sum_{i=1}^N \omega_i' (ue_i), \]
  where $\omega_1', \ldots, \omega_N'$ are independent standard normal random variables independent from $X^{(A)}$. Once we see a realization $X(\omega)$ of $X$ we do not know a priori, which values the $\omega_i'$'s attained. But we can calculate them from the fact that
  \[ 0 = a_j(X^{(A)}) = a_j(X) - \sum_{i=1}^N \omega_i a_j(ue_i) \]
  for all $1 \leq j \leq N$, which leads to the system of linear equations $B \xi = b(X)$, its solution $\xi(X)$ and the claimed representation for $X^{(A)}$.
\end{proof}


\section{Equivalence of measures}\label{S:EQU_MEAS}
 
Let $X$ be a continuous Gaussian process and let $X^{(A)}$ be the conditioned process of $X$ with respect to a finite set of conditions $A = \{ a_1, \ldots, a_N \}$. Moreover, let $P_X$ and $P_{X^{(A)}}$ be the induced measures of $X$ and $X^{(A)}$ on $(C([0,T]), \mathcal{C})$.
 
We can not expect that $P_X$ and $P_{X^{(A)}}$ are equivalent on $\mathcal{C}$ since
\[ P_X(\{ f \in C([0,T]) : a(f) = 0 \quad \forall a \in A \}) = 0 \]
in case that $X$ does not fulfill all conditions in $A$, while
\[ P_{X^{(A)}}(\{ f \in C([0,T]) : a(f) = 0 \quad \forall a \in A \}) = 1. \]
However, in this section, we show that $P_X$ and $P_{X^{(A)}}$ are equivalent on a suitable sub-$\sigma$-algebra of $\mathcal{C}$.

Let $X$ be generated by the operator $u:H \rightarrow C([0,T])$ and let $\{ e_1, \ldots, e_N \}$ be an orthonormal basis in the detached Hilbert space $H_{(A)} \subset H$ (w.l.o.g. we assume $\dim (H_{(A)}) = N$; otherwise let some of the $e_i$'s be $0$). Recall that $\F_s \subset \mathcal{C}$ is the smallest $\sigma$-algebra on $C([0,T])$ such that all $\delta_r$, $0 \leq r \leq s$, are $\F_s$-$\B(\R)$-measurable.
 
\begin{theorem}\label{T:EQU_MEASURES}
  The probability measures $\Prob_X$ and $\Prob_{X^{(A)}}$ are equivalent on $\F_s$ if and only if
  \equ[E:EQU_MEAS_1]{ \text{there exist $e_i' \in H^{(A)}$, $1 \leq i \leq N,\quad$ such that $(ue_i')(x) = (ue_i)(x), \quad 0 \leq x \leq s$.} }
  Otherwise $\Prob_X$ and $\Prob_{X^{(A)}}$ are orthogonal on $\F_s$.
\end{theorem}

We will prove the different assertions of Theorem~\ref{T:EQU_MEASURES} in the subsequent sections. We start by introducing some additional notation. For $d \geq 1$, let $\Prob_d$ be the standard Gaussian law on $(\R^d, \B(\R^d))$, i.e., $\Prob_d = \bigotimes_{i=1}^d \Prob_1$, where $\Prob_1$ is the standard normal law on $\R$, and consider the probability space $(\Omega, \A, \Prob)$ with $\Omega = \bigotimes_{i=1}^\infty \R$, $\A = \bigotimes_{i=1}^\infty \B(\R)$, and $\Prob = \bigotimes_{i=1}^\infty \Prob_1$. 

We are only interested in the laws of $X$ and $X^{(A)}$ and might thus, without loss of generality, assume that they are defined on the probability space $(\Omega, \A, \Prob)$. Let $\{ f_1, f_2, \ldots \}$ be an orthonormal basis in the reduced Hilbert space $H^{(A)} \subset H$. We may write $X: \Omega \rightarrow C([0,T])$ as
\equ[E:EQU_MEAS_4]{ X(\omega) = X(\omega_1, \omega_2, \dots) = \sum_{i=1}^N \omega_i (ue_i) + \sum_{j=1}^\infty \omega_{j+N} (uf_j). }

\subsection{If~\eqref{E:EQU_MEAS_1} then $\Prob_X \ll \Prob_{X^{(A)}}$ on $\F_s$}

Consider $X^{(A)}: \Omega \rightarrow C([0,T])$ defined as
\equ[E:EQU_MEAS_5]{ X^{(A)}(\omega) = X^{(A)}(\omega_1, \omega_2, \dots) = \sum_{j=1}^\infty \omega_j (uf_j). }
Given the $e_i$'s as in~\eqref{E:EQU_MEAS_1} define $\xi_{ij} = \la e_i', f_j \ra$, $1 \leq i \leq N$, $j \geq 1$. Then $(\xi_{ij})_{i,j=1}^\infty$ fulfills
\[ \sum_{j=1}^\infty \xi_{ij}^2 < \infty, \quad \text{for all $i \geq 1$,} \]
and we have
\[ (u e_i) = (u e_i') = \sum_{j=1}^\infty \xi_{ij} (u f_j) \]
on $[0,s]$. Consider the mapping $M: \Omega \rightarrow \Omega$ defined by
\begin{align}
  \notag (\omega_1, \omega_2, \dots) &\mapsto (\omega_1', \omega_2', \dots), \\
  \label{E:EQU_MEAS_2} \omega_j' &= \sum_{i=1}^N \omega_i \xi_{ij} + \omega_{j+N}.
\end{align}
From~\eqref{E:EQU_MEAS_4}, \eqref{E:EQU_MEAS_5}, and~\eqref{E:EQU_MEAS_2}, we obtain on $[0,s]$
\[ X(\omega) = \sum_{j=1}^\infty \left( \sum_{i=1}^N \omega_i \xi_{ij} + \omega_{j+N}\right) (uf_j) = X^{(A)}(M(\omega)). \]

\begin{prop}\label{P:EQU_MEAS}
  For $F \in \A$ with $\Prob(F) > 0$ it holds $\Prob(M(F)) > 0$.
\end{prop}

\begin{proof}
  Let $F \in \A$ with $\Prob(F) > 0$. Note that $\R^N \times \Omega = \Omega$. For an element $x \in \R^N$ define
  \al{ F_x &= \{ y \in \Omega : (x,y) \in F \} \subset \Omega, \\
      F_x' &= \{ (x,y) \in \Omega : y \in F_x \} \subset \Omega. }
  Then
  \[ 0 < \Prob(F) = \int_{\R^N} \Prob(F_x) \, \Prob_N(dx) \]
  which implies the existence of a $z = (z_1, \ldots, z_N) \in \R^N$ with $\Prob(F_z) > 0$. Define the element $l(z) = (l_1, l_2, \ldots) \in \Omega$ by $l_j = \sum_{i=1}^N \xi_{ij} z_i$. By Jensen's inequality,
  \[ \sum_{j=1}^\infty l_j^2 = \sum_{j=1}^\infty \left(\sum_{i=1}^N \xi_{ij} z_i\right)^2 \leq N \sum_{j=1}^\infty \sum_{i=1}^N \xi_{ij}^2 z_i^2 = N \sum_{i=1}^N z_i^2 \sum_{j=1}^\infty \xi_{ij}^2 < \infty \]
  and thus $l(z) \in l_2$. Define $\tau_{l(z)}: \Omega \rightarrow \Omega$ by $\tau_{l(z)}(\omega) = \omega - l(z)$. Then, for the subset $F_z' \subset F$ it holds $M(F_z') = l(z) + F_z = \tau_{l(z)}^{-1}(F_z)$ and thus
  \equ[E:EQU_MEAS_3]{ \Prob(M(F)) \geq \Prob(M(F_z')) = \Prob(\tau_{l(z)}^{-1}(F_z)) = \Prob \circ \tau_{l(z)}^{-1}(F_z). }
  The probability space $(\Omega, \A, \Prob)$ is the canonical model for the Gaussian process $Z = (Z_n)_{n \in \N}$ with covariance $\E Z_m Z_n = \delta_{m,n}$, $m,n \in \N$. The Cameron-Martin space associated with $Z$ is $l_2$ and thus, since $l(z) \in l_2$, the probability measures $\Prob$ and $\Prob \circ \tau_{l(z)}^{-1}$ are equivalent (Theorem~14.17 in~\cite{Jan97}). Hence, since $\Prob(F_z) > 0$ we have, by~\eqref{E:EQU_MEAS_3},
  \[ \Prob(M(F)) \geq \Prob \circ \tau_{l(z)}^{-1}(F_z) > 0. \qedhere \]
\end{proof}

Now, let $F \in \F_s$ with $\Prob_X(F) = \Prob(X^{-1}(F)) > 0$. Since $M$ is surjective,
\al{ \Prob_{X^{(A)}}(F) = \Prob(\{ \omega' \in \Omega : X^{(A)}(\omega') \in F\}) = \Prob(\{ M(\omega) \in \Omega : \omega \in \Omega \text{ and } X^{(A)}(M(\omega)) \in F\}). }
Because of $F \in \F_s$ and $X^{(A)}(M(\omega)) = X(\omega)$ on $[0,s]$,
\al{ \Prob_{X^{(A)}}(F) = \Prob(\{ M(\omega) \in \Omega : \omega \in \Omega \text{ and } X(\omega) \in F\}) = \Prob(M(X^{-1}(F))). }
Since $\Prob(X^{-1}(F)) > 0$, Proposition~\ref{P:EQU_MEAS} yields $\Prob(M(X^{-1}(F))) > 0$ and thus $\Prob_{X^{(A)}}(F) > 0$. We thus have proven that~\eqref{E:EQU_MEAS_1} implies $\Prob_X \ll \Prob_{X^{(A)}}$ on $\F_s$.

\subsection{If~\eqref{E:EQU_MEAS_1} then $\Prob_{X^{(A)}} \ll \Prob_X$ on $\F_s$}

We proceed in a similar way as in the previous section. Consider $X^{(A)}: \Omega \rightarrow C([0,T])$ defined as
\equ[E:EQU_MEAS_5b]{ X^{(A)}(\omega) = X^{(A)}(\omega_1, \omega_2, \dots) = \sum_{j=1}^\infty \omega_{N+j} (uf_j). }
Let $\xi_{ij}$ be defined as before and consider the mapping $N: \Omega \rightarrow \Omega$ defined by
\begin{align}
  \notag (\omega_1, \omega_2, \dots) &\mapsto (\omega_1', \omega_2', \dots), \\
  \notag \omega_j' &= \omega_j, \qquad &\text{for $1 \leq j \leq N$,} \\
  \label{E:EQU_MEAS_2b} \omega_j' &= - \sum_{i=1}^N \omega_i \xi_{i(j-N)} + \omega_{j}, \qquad &\text{for $j \geq N+1$.}
\end{align}
From~\eqref{E:EQU_MEAS_4}, \eqref{E:EQU_MEAS_5b}, and~\eqref{E:EQU_MEAS_2b}, we obtain on $[0,s]$
\[ X(N(\omega)) =  \sum_{i=1}^N \omega_i (ue_i) + \sum_{j=1}^\infty \left(- \sum_{i=1}^N \omega_i \xi_{ij} + \omega_{N+j}\right) (uf_j) = X^{(A)}(\omega). \]

\begin{prop}\label{P:EQU_MEASb}
  For $F \in \A$ with $\Prob(F) > 0$ it holds $\Prob(N(F)) > 0$.
\end{prop}

\begin{proof}
  With the notation of the proof of Proposition~\ref{P:EQU_MEAS}, we have
  \[ 0 < \Prob(F) = \int_{\R^N} \Prob(F_x) \, \Prob_N(dx) = \int_{\{x \in \R^N : \ \Prob(F_x) > 0 \}} \Prob(F_x) \, \Prob_N(dx). \]
  By the Cameron-Martin Theorem it follows for every $x \in \R^d$ with $\Prob(F_x) > 0$ that $\Prob \circ \tau_{-l(x)}^{-1} (F_x) > 0$ and thus
  \[ 0 < \int_{\{x \in \R^N : \ \Prob(F_x) > 0 \}} \Prob \circ \tau_{-l(x)}^{-1} (F_x)) \, \Prob_N(dx) \leq \Prob(N(F)). \qedhere \]
\end{proof}

Now, let $F \in \F_s$ with $\Prob_{X^{(A)}}(F) = \Prob([X^{(A)}]^{-1}(F)) > 0$. This implies by Proposition~\ref{P:EQU_MEASb} that $\Prob(N([X^{(A)}]^{-1}(F))) > 0$. Because of $F \in \F_s$ and $X(N(\omega)) = X^{(A)}(\omega)$ on $[0,s]$, it follows that $N([X^{(A)}]^{-1}(F)) \subset X^{-1}(F)$ and thus $\Prob(X^{-1}(F)) = \Prob_X(F) > 0$. We thus have proven that~\eqref{E:EQU_MEAS_1} implies $\Prob_{X^{(A)}} \ll \Prob_X$ on $\F_s$.

\subsection{If not~\eqref{E:EQU_MEAS_1} then $\Prob_X$ and $\Prob_{X^{(A)}}$ are orthogonal on $\F_s$}
  We assume that $N=1$. By doing so we do not loose any generality since we could as well impose the conditions one by one and build in this way a cascade of on $\F_s$ equivalent measures. Fix $0 \leq s \leq T$ and set $e = e_1$. Define $u_s: H \rightarrow C([0,s])$ by $(u_sh)(x) = (uh)(x)$, $h \in H$, $x \in [0,s]$, and assume that for all $e' \in H^{(A)}$ there is an $x \in [0,s]$ such that $(ue)(x) \neq (ue')(x)$ which implies $e-e' \notin \ker(u_s)$, the kernel of $u_s$. Since elements in $H^{(A)}$ are orthogonal to $e$ and $e \notin \ker(u_s)$, it follows that $e$ is orthogonal to $\ker(u_s) \subset H$. The orthogonal complement of $\ker(u_s)$ is equal to the closed image of the adjoint operator $u_s^*$. This implies that there is a sequence of functionals $(b_n)_{n=1}^\infty \subset C([0,s])^*$ such that $u_s^* b_n \rightarrow e$ and $(u_s^{(A)})^* b_n \rightarrow 0$, where $u_s^{(A)}: H^{(A)} \rightarrow C([0,s])$ is the restriction of $u_s$ to $H^{(A)}$. We may assume that $\| (u_s^{(A)})^* b_n \| \leq 2^{-n}$ (by choosing a suitable sub-sequence of $(b_n)_{n=1}^\infty$ if necessary). Now set $\tilde b_n = b_n / \|(u_s^{(A)})^* b_n\|^{1/2}$. Then $\tilde b_n(X)$ and $\tilde b_n(X^{(A)})$ are Gaussian random variables, which, by~\eqref{E:CP_4}, satisfy
  \[ \E \left[ \tilde b_n(X)^2 \right] = \| u^* b_n \|^2 / \|(u_s^{(A)})^* b_n\| \rightarrow \infty, \]
  as $n \rightarrow \infty$, and
  \[ \E \left[ \tilde b_n(X^{(A)})^2 \right] = \| (u^{(A)})^* b_n \|^2 / \|(u_s^{(A)})^* b_n\| \leq 2^{-n}. \]
  From this it follows by the Borel-Cantelli Lemma that, almost surely, $\limsup_{n \rightarrow \infty} |\tilde b_n(X)| = \infty$ and $\lim_{n \rightarrow \infty} |\tilde b_n(X^{(A)})| = 0$. Hence, $X$ and $X^{(A)}$ induce orthogonal laws on $C([0, s])$ which implies that $\Prob_X$ and $\Prob_{X^{(A)}}$ are orthogonal on $\F_s$.

%


\section{Non-anticipative representations}\label{S:NON_ANTI}

Now, we consider alternative, non-anticipative representations for $X^{(A)}$ in the same setting as in the previous section. We assume that the supremum over all $0 \leq s \leq T$ for which~\eqref{E:EQU_MEAS_1} holds is $T$. If this is not the case, the following calculations can only be performed on an interval $[0, T_*) \subset [0,T)$.

Recall that a progressively measurable functional on $C([0,T])$ is a mapping $\beta: [0, T] \times C([0,T]) \rightarrow \R$ such that for each $0 \leq s \leq T$, the restriction of $\beta$ to $[0,s] \times C([0,T])$ is $\B([0, s]) \otimes \F_s$-$\B(\R)$-measurable.

Let $W=(W_s)_{s \in [0,T]}$ be a standard linear Brownian motion defined on a probability space $(\Omega, \A, \Prob)$ and assume that there is a $0 \neq \alpha \in \R$ and a progressively measurable functional $\beta$ on $C([0,T])$ with
\equ[E:NON_ANTI_2]{ \int_0^S |\beta(x,X)| \, dx < \infty }
$\Prob$-almost surely for all $S < T$, such that $X$ is a (strong) solution of the stochastic differential equation
\equ[E:NON_ANTI_1]{ dX_s = \alpha dW_s + \beta(s,X) ds, \qquad X_0 = 0, \qquad 0 \leq s < T. }

In order to apply the results from the previous section, it proves to be useful to assume without loss of generality $\Omega = C([0,T])$ (recall that we do not distinguish between Gaussian processes with the same law): by~\eqref{E:NON_ANTI_1} and since $\alpha \neq 0$,
  \[ W_s = \alpha^{-1} X_s - \alpha^{-1} \int_0^s \beta(x,X) \, dx, \qquad 0 \leq s < T. \]
  Let $\Prob_X$ be the induced measure of $X$ on the space $(C([0,T]), \mathcal{C})$. Define the processes $\widehat{X}: (C([0,T]), \mathcal{C}) \rightarrow (C([0,T]), \mathcal{C})$ and $\widehat{W}: (C([0,T]), \mathcal{C}) \rightarrow (C([0,T]), \mathcal{C})$ by $(\widehat{X} f)(s) = f(s)$ and
  \[ (\widehat{W} f)(s) = \alpha^{-1} (\widehat{X} f)(s) - \alpha^{-1} \int_0^s \beta(x,\widehat{X} f) \, dx \]
  for $0 \leq s < T$ and $f \in C([0,T])$. Then, on $(C([0,T]), \mathcal{C}, \Prob_X)$, $\widehat{W}$ is a standard Brownian motion, $\widehat{X} = X$ in distribution, and we have
  \[ d\widehat{X}_s = \alpha d\widehat{W}_s + \beta(s,\widehat{X}) ds, \qquad \widehat{X}_0 = 0, \qquad 0 \leq s < T\]
  with
  \[ \int_0^S |\beta(x,\widehat{X})| \, dx < \infty \]
  $\Prob_X$-almost surely for all $S < T$. From the construction follows that the natural filtration of $\widehat{W}$ and $\widehat{X}$ is $\F$.

\subsection{Existence of a describing SDE}

Let $\Prob_{X^{(A)}}$ be the induced measure of $X^{(A)}$ on $(C([0,T]), \mathcal{C})$.

\begin{theorem}\label{T:SDE_EX}
  There is a Brownian motion $W' = (W_s')_{s \in [0,T]}$ defined on the probability space $(C([0,T]), \mathcal{C}, \Prob_{X^{(A)}})$ and a progressively measurable functional $\delta$ on $C([0,T])$ with
  \equ[E:NON_ANTI_6]{ \int_0^S |\delta(x, X^{(A)})| \, dx < \infty }
  $\Prob_{X^{(A)}}$-almost surely for all $S < T$ such that the conditioned process $X^{(A)}$ is a (strong) solution of the stochastic differential equation
  \equ[E:NON_ANTI_4]{ d X^{(A)}_s = \alpha dW'_s + \delta(s, X^{(A)}) ds, \qquad X^{(A)}_0 = 0, \qquad 0 \leq s < T. }
\end{theorem}

\begin{proof}
  We consider the mapping $Y: (C([0,T]), \mathcal{C}) \rightarrow (C([0,T]), \mathcal{C})$ defined by $Y(f) = f$ for $f \in C([0,T])$. Then, under the measure $\Prob_X$, the law of $Y$ is the same as the law of $X$ and under the measure $\Prob_{X^{(A)}}$, the law of $Y$ is the same as the law of $X^{(A)}$. Under $\Prob_X$, the semimartingale $Y=(Y_s)_{s \in [0,T)}$ has the decomposition $Y = M + A$, where $M$ is a continuous martingale and $A$ a finite variation process,
  \[ M_s = \alpha W_s, \qquad A_s = \int_0^s \beta(x, Y) \, dx. \]
  By Theorem~\ref{T:EQU_MEASURES} the measures $\Prob_X$ and $\Prob_{X^{(A)}}$ are equivalent on $\F_s$ for all $0 \leq s < T$. Hence,
  \[ Z_s = \E_{\Prob_X} \left[ \frac{d\Prob_{X^{(A)}}}{d\Prob_X} \, \Big| \, \F_s \right], \quad 0 \leq s < T, \]
  is an almost sure non-negative continuous $(\Prob_X, \F)$-martingale. By Girsanov's Theorem (see e.g.\ Theorem~III.35 in~\cite{Pro05}), $Y$ is a semimartingale under $\Prob_{X^{(A)}}$ with decomposition $Y = L + C$ with
  \equ[E:NON_ANTI_5]{ L_s = M_s - \int_0^s Z_x^{-1} \, d[Z,M]_x }
  being a local martingale under $\Prob_{X^{(A)}}$, where $[Z,M]$ denotes the quadratic covariation process of $M$ and $Z$, and $C = Y - L$ is a $\Prob_{X^{(A)}}$-finite variation process. By the martingale representation theorem (see e.g.~Theorem~4.3.4 in~\cite{Oks07}) there is an adapted stochastic process $\gamma$ such that
  \[ Z_s = \int_0^s \gamma(x) \, dW_x \qquad \text{and} \qquad \E_{\Prob_X} \left[ \int_0^{s} \gamma^2(x) \, dx \right] < \infty. \]
  Since $M = \alpha W$ it follows $d[Z,M]_x = \alpha \gamma(x) dx$ under $\Prob_X$ and thus under $\Prob_{X^{(A)}}$. Hence, by~\eqref{E:NON_ANTI_5},
  \al{ Y_s &= L_s + (Y_s - L_s) = L_s + (M_s + A_s - L_s) \\
           &= \alpha \left( W_s - \int_0^s Z_x^{-1} \gamma(x) \, dx \right) + \left( \int_0^s \left( \alpha Z_x^{-1} \gamma(x) + \beta(x, X) \right) \, dx \right). }
  The quadratic variation process of the first bracket is $s$ under $\Prob_X$ and thus under $\Prob_{X^{(A)}}$. By L\'evy's characterization of Brownian motion,
  \equ[E:NON_ANTI_5b]{ W'_s = W_s - \int_0^s Z_x^{-1} \gamma(x) \, dx }
  is a Brownian motion under $\Prob_{X^{(A)}}$. That is,
  \[ Y_s = \alpha W'_s + \int_0^s ( \alpha Z_x^{-1} \gamma(x) + \beta(x, Y) ) \, dx, \qquad 0 \leq s < T. \]
  Since the natural filtration of $Y$ is $\F$ and the process $( \alpha Z_x^{-1} \gamma(x) + \beta(x, Y) )_{0 \leq s < T}$ is adapted to this filtration we have
  \[ \alpha Z_x^{-1} \gamma(x) + \beta(x, Y) = \delta(x, Y) \]
  for some progressively measurable functional $\delta$ on $C([0,T])$. Moreover, from~\eqref{E:NON_ANTI_2} and~\eqref{E:NON_ANTI_5b} it follows
  \[ \int_0^S |\delta(x, Y)| \, dx < \infty \]
  $\Prob_{X^{(A)}}$-almost surely for all $S < T$.
\end{proof}

\subsection{Determination of the drift}

Theorem~\ref{T:SDE_EX} provides us with a progressively measurable functional $\delta$ on $C([0,T])$ for which 
\[ \int_0^S |\delta(x, X^{(A)})| \, dx < \infty \]
almost surely for all $S < T$. But in the following we need more than this.

\begin{prop}\label{P:FINITE_INT}
  The progressively measurable functional $\delta$ in Theorem~\ref{T:SDE_EX} satisfies
  \[ \E \int_0^S |\delta(x, X^{(A)})| \, dx < \infty, \quad S < T. \]
\end{prop}

\begin{proof}
  From~\eqref{E:NON_ANTI_6} we know $|\delta(s, X^{(A)})| < \infty$ almost surely for almost all $0 \leq s \leq S$ and thus the limit in
  \al{ \delta(s, X^{(A)}) &= \lim_{\eps \searrow 0} \eps^{-1} \int_s^{s+\eps} \delta(x, X^{(A)}) \, dx = \lim_{\eps \searrow 0} \eps^{-1} \left( X^{(A)}_{s+\eps} - X^{(A)}_s - \alpha W'_{s+\eps} + \alpha W'_s \right) }
  exists and is, as the limit of Gaussian random variables, a Gaussian random variable.

  Let $\sigma^2(x) = \E |\delta(x, X^{(A)})|^2$  be the variance of $\delta(x, X^{(A)})$ and for $n \in \N$ set $\delta_n(x) = \min\{ 1, n/\sigma(x) \} \delta(x, X^{(A)})$. Then
  \[ \sigma_n^2(x) = \E |\delta_n(x)|^2 = \min\{ \sigma^2(x), n^2 \} \leq n^2 \]
  and $\sigma_n^2(x) \nearrow \sigma^2(x)$ for all $x$ as $n \rightarrow \infty$. Since $\delta_n(x)$ is Gaussian we have $\E |\delta_n(x)| = \sqrt{2/\pi} \sigma_n(x)$ and by the Cauchy-Schwartz inequality
  \[ \E |\delta_n(x) \delta_n(y)| \leq \sqrt{ \E |\delta_n(x)|^2 \E |\delta_n(y)|^2 } = \sigma_n(x) \sigma_n(y). \]
  Define 
  \[ Z = \int_0^S |\delta(x, X^{(A)})| \, dx \qquad \text{and} \qquad Z_n = \int_0^S |\delta_n(x)| \, dx. \]
  Then we have $Z_n \leq Z$ and
  \[ \E Z_n = \int_0^S \E |\delta_n(x)| \, dx = \sqrt{\frac{2}{\pi}} \int_0^S \sigma_n(x) \, dx \nearrow \sqrt{\frac{2}{\pi}} \int_0^S \sigma(x) \, dx \]
  as $n \rightarrow \infty$. Moreover,
  \al{ \E Z_n^2 &= \E \int_0^S \int_0^S \delta_n(x) \delta_n(y) \, dx \, dy
         \leq \int_0^S \int_0^S \E | \delta_n(x) \delta_n(y) | \, dx \, dy \\
         &\leq \int_0^S \int_0^S \sigma_n(x) \sigma_n(y) \, dx \, dy
         = \left(\int_0^S \sigma_n(x) dx \right)^2. }
  Thus, for the variance $\V Z_n = \E Z_n^2 - (\E Z_n)^2$,
  \al{ \V Z_n &= \left(\int_0^S \sigma_n(x) \, dx \right)^2 - \left( \sqrt{\frac{2}{\pi}} \int_0^S \sigma_n(x) \, dx \right)^2 = (1 - 2/\pi)\left(\int_0^S \sigma_n(x) \, dx \right)^2. }
  Since $Z_n \leq Z$ it follows for $\eps > 0$ 
  \al{ \Prob\left(Z \leq \eps \int_0^S \sigma_n(x) \, dx \right) &\leq \Prob\left(Z_n \leq \eps \int_0^S \sigma_n(x) \, dx \right) \\
         &= \Prob\left(\E Z_n - Z_n \geq \E Z_n - \eps \int_0^S \sigma_n(x) \, dx \right) \\
         &\leq \Prob\left(|\E Z_n - Z_n| \geq (\sqrt{2/\pi} - \eps ) \int_0^S \sigma_n(x) \, dx \right). }
  By Chebyshev's inequality,
  \al{ \Prob\left(Z \leq \eps \int_0^S \sigma_n(x) \, dx \right) &\leq \frac{\V Z_n}{(\sqrt{2/\pi} - \eps )^2 \left(\int_0^S \sigma_n(x) \, dx \right)^2} \\
         &= \frac{(1 - 2/\pi)\left(\int_0^S \sigma_n(x) \, dx \right)^2}{(\sqrt{2/\pi} - \eps )^2 \left(\int_0^S \sigma_n(x) \, dx \right)^2} \\
         &= \frac{1 - 2/\pi}{(\sqrt{2/\pi} - \eps )^2}. }
  Thus, for $\eps > 0$ small enough,
  \[ \Prob\left(Z \leq \eps \int_0^S \sigma_n(x) \, dx \right) \leq c < 1. \]
  Note that the constant $c$ depends only on $\eps$ but not on $n \in \N$. Hence, taking the limit $n \rightarrow \infty$, we obtain by the monotone convergence theorem
  \al{ 0 < \Prob\left(Z > \eps \int_0^S \sigma(x) \, dx \right) =  \Prob\left(\eps^{-1} \int_0^S |\delta(x, X^{(A)})| \, dx > \int_0^S \sigma(x) \, dx \right). }
  Since $\int_0^S |\delta(x, X^{(A)})| \, dx < \infty$ almost surely it follows $\int_0^S \sigma(x) \, dx < \infty$ and finally
  \[ \E \int_0^S |\delta(x, X^{(A)})| \, dx = \sqrt{\frac{2}{\pi}} \int_0^S \sigma(x) \, dx < \infty. \qedhere \]
\end{proof}

\begin{theorem}\label{T:DRIFT}
  Almost surely, for almost all\, $0 \leq s < T$, the drift term $\delta(s,X^{(A)})$ in Theorem~\ref{T:SDE_EX} is
  \[ \delta(s,X^{(A)}) = \lim_{r \searrow 0} \frac{\E [ X^{(A)}_{s+r} \mid \F_s] - X_s^{(A)}}{r}. \]
\end{theorem}

\begin{proof}
  Let $s \geq 0$ be fixed. By~\eqref{E:NON_ANTI_4}, for $r > 0$,
  \[ X^{(A)}_{s+r} = X^{(A)}_s + \alpha W'_{s+r} - \alpha W'_s + \int_s^{s+r} \delta(x,X^{(A)}) \, dx. \]
  Hence, since $X^{(A)}_s$ is $\F_s$-measurable,
  \[ \E [ X^{(A)}_{s+r} \mid \F_s] = X^{(A)}_s + \E [ \alpha W'_{s+r} - \alpha W'_s \mid \F_s ] + \E \left[ \int_s^{s+r} \delta(x,X^{(A)}) \, dx \,\Big|\, \F_s \right]. \]
  Since $W'$ has independent increments with mean $0$, the second term vanishes. By Proposition~\ref{P:FINITE_INT} we can apply Fubini's theorem to the third term and get
  \[ \E [ X^{(A)}_{s+r} \mid \F_s] = X^{(A)}_s + \int_s^{s+r} \E [ \delta(x,X^{(A)}) \mid \F_s ] \, dx. \]
  Finally (see e.g.~Corollary 2.14 in~\cite{Mat95}),
  \al{ \lim_{r \searrow 0} \frac{\E [ X^{(A)}_{s+r} \mid \F_s] - X^{(A)}_s}{r} = \lim_{r \searrow 0} \frac{1}{r} \int_s^{s+r} \E [ \delta(x,X^{A}) \mid \F_s ] \, dx = \E [ \delta(s,X^{(A)}) \mid \F_s ] = \delta(s,X^{(A)}) }
  for almost all $s \geq 0$.
\end{proof}


\section{The Markov property and the expected future}\label{S:MARKOV}

In this section we assume that the Gaussian process $X = (X_s)_{s \in [0,T]}$ is a Markov process with $R_X(s,t) \neq 0$ for all $0 < s,t < T$. Let $X^{(A)} = (X^{(A)}_s)_{s \in [0,T]}$ be the conditioned process of $X$ with respect to $A = \{ a_1, \ldots, a_N \}$ and let $\F^{X^{(A)}} = (\F^{X^{(A)}}_s)_{s \in [0,T]}$ be the natural filtration of $X^{(A)}$. The process $X^{(A)}$ is in general not a Markov process as well.

\subsection{Retrieving the Markov property}

Define Gaussian processes $I^{{(A)}, i}$ by
\[ I^{{(A)}, i}_s = \int_0^s X^{(A)}_x \, a_i(dx), \quad 0 \leq s \leq T, \quad 1 \leq i \leq N. \]

\begin{theorem}\label{T:MP}
  The Gaussian process $(X^{(A)}, I^{{(A)}, 1}, \ldots, I^{{(A)}, N})$ is an $(N+1)$-dimensional (in general time-inhomogeneous) Markov process.
\end{theorem}

First, we show the result for the case that $X$ is Brownian motion and then in the general case.

\begin{proof}[Proof of Theorem~\ref{T:MP} for $X$ Brownian motion]
  We assume that $a_1(X), \ldots, a_n(X)$ are independent standard normal random variables. Without loss of generality we can do so by Proposition~\ref{P:ON_COND}. For every $0 \leq s \leq t \leq T$ we define the Gaussian random variable $Z_{s,t}$ by
  \[ Z_{s,t} = X^{(A)}_t- \E[X^{(A)}_t | \F^{X^{(A)}}_s]. \]
  Then $Z_{s,t}$ is independent from $\F^{X^{(A)}}_s$. We show that
  \equ[E:MP_9]{ Z_{s,t} = X^{(A)}_t - \E[X^{(A)}_t | \{ X^{(A)}_s, I^{{(A)}, 1}_s, \ldots, I^{{(A)}, N}_s \} ], }
  which implies that $\E[X^{(A)}_t | \F^{X^{(A)}}_s ] = \E[X^{(A)}_t | \{ X^{(A)}_s, I^{{(A)}, 1}_s, \ldots, I^{{(A)}, N}_s \} ]$. Since the natural filtration of $X^{(A)}$ and $(X^{(A)}, I^{{(A)}, 1}, \ldots, I^{{(A)}, N})$ coincide, this will prove the theorem.

  Set $\psi_i(y,s) = a_i(\Ind_{[y,s]})$ and rewrite the Gaussian processes $I^i$ in~\eqref{E:INT_PROC} as
  \begin{align}
    I^i_s = \int_0^s X_x \, a_i(dx) = \int_0^s \int_0^x \, dX_y \, a_i(dx) = \int_0^s \int_y^s \, a_i(dx) \, dX_y  = \int_0^s \psi_i(y,s) \, dX_y, \label{E:MP_5}
  \end{align}
  $0 \leq s \leq T$, $1 \leq i \leq N$. We condition the process $X$ on $a_i(X) = I^i_T = 0$ almost surely for $1 \leq i \leq N$. Since we assume $I^1_T, \ldots, I^N_T$ to be independent random variables with $\E \left[ I^1_T \right]^2 = 1$, the conditioned process $X^{(A)}$ and the processes $I^{{(A)}, i}$ are (as in Proposition~\ref{P:ANTI}) given by
  \equ[E:MP_8]{ X^{(A)}_s = X_s - \sum_{j=1}^N I^j_T \E \left[X_s I^j_T \right]\quad \text{and} \quad I^{{(A)}, i}_s = I^i_s - \sum_{j=1}^N I^j_T \E \left[I^i_s I^j_T\right], }
  $0 \leq s \leq T$, $1 \leq i \leq N$. Now, define Gaussian processes $J^i$ and $J^{{(A)}, i}$ by
  \begin{align}
    J^i_s = \psi_i(s,T) X_s + I^i_s = \int_0^s ( \psi_i(s,T) + \psi_i(y,s) ) \, dX_y 
          = \int_0^s \psi_i(y,T) \, dX_y, \label{E:MP_6}
  \end{align}
  and
  \begin{align}
    J^{{(A)}, i}_s &= \psi_i(s,T) X^{(A)}_s + I^{{(A)}, i}_s \label{E:MP_7} \\
                   &= \psi_i(s,T) \left( X_s - \sum_{j=1}^N I^j_T \E \left[X_s I^j_T \right]\right) + I^i_s - \sum_{j=1}^N I^j_T \E \left[I^i_s I^j_T \right]\notag \\
                   &= \psi_i(s,T) X_s + I^i_s - \sum_{j=1}^N I^j_T \E \left[(\psi_i(s,T) X_s + I^i_s) I^j_T\right] \notag \\
                   &= J^i_s - \sum_{j=1}^N I^j_T \E \left[J^i_s I^j_T\right], \quad 0 \leq s \leq T, 1 \leq i \leq N. \label{E:MP_7b} 
  \end{align}

  By~\eqref{E:MP_7}, it is enough to show
  \[ Z_{s,t} = X^{(A)}_t - \E[X^{(A)}_t | \{ X^{(A)}_s, J^{{(A)}, 1}_s, \ldots, J^{{(A)}, N}_s \} ] \]
  in order to show~\eqref{E:MP_9}. Define
  \equ[E:MP_Z]{ Z^*_{s,t} = X^{(A)}_t - X^{(A)}_s - \sum_{i=1}^N b_i(s,t) J^{{(A)}, i}_s, }
  where the $b_i$'s are chosen such that $Z^*_{s,t}$ is independent from $J^{{(A)}, i}_s$, $1 \leq i \leq N$, i.e., we require
  \begin{align}
    0 = \E \left[Z^*_{s,t} J^{{(A)}, i}_s \right]&= \E \left[Z^*_{s,t} \left( J^i_s - \sum_{j=1}^N I^j_T \E \left[ J^i_s I^j_T \right] \right) \right]\label{E:MP_10} \\
         &= \E \left[Z^*_{s,t} J^i_s\right] - \sum_{j=1}^N \E \left[Z^*_{s,t} I^j_T\right] \E \left[J^i_s I^j_T \right]. \notag
  \end{align}
  By~\eqref{E:MP_8} and~\eqref{E:MP_7b},
  \begin{align}
    Z^*_{s,t} &= X_t - \sum_{i=1}^N I^i_T \E \left[X_t I^i_T \right] - X_s + \sum_{i=1}^N I^i_T \E \left[X_s I^i_T \right] \label{E:MP_10b} \\
              &\qquad \qquad - \sum_{i=1}^N b_i(s,t) \left( J^i_s - \sum_{j=1}^N I^j_T \E \left[J^i_s I^j_T \right]\right) \notag \\
              &= X_t - X_s - \sum_{i=1}^N b_i(s,t) J^i_s - \sum_{i=1}^N I^i_T \E \left[ \left( X_t - X_s - \sum_{j=1}^N b_j(s,t) J^j_s \right) I^i_T \right], \notag
  \end{align}
  and thus
  \begin{align}
    \E \left[Z^*_{s,t} I^j_T\right] &= \E \left[\left( X_t - X_s - \sum_{i=1}^N b_i(s,t) J^i_s \right) I^j_T\right] \label{E:MP_11} \\
       &\qquad - \sum_{i=1}^N \E \left[I^i_T I^j_T\right] \E \left[ \left( X_t - X_s - \sum_{k=1}^N b_k(s,t) J^k_s \right) I^i_T\right] = 0, \notag
  \end{align}
  since we assumed $\E I^i_T I^j_T = \delta_{i,j}$. Moreover,
  \al{ \E \left[Z^*_{s,t} J^i_s\right] &= \E \left[(X_t - X_s) J^i_s\right] - \sum_{j=1}^N b_j(s,t) \E \left[J^j_s J^i_s\right] \\
         &\qquad \qquad - \sum_{j=1}^N \E \left[\left( X_t - X_s - \sum_{k=1}^N b_k(s,t) J^k_s \right) I^j_T \right] \E \left[ I^j_T J^i_s\right], }
  where $\E \left[(X_t - X_s) J^i_s\right] = 0$ and $\E \left[I^j_T J^i_s\right] = \E \left[J^j_s J^i_s\right]$ by~\eqref{E:MP_5} and~\eqref{E:MP_6}. Hence, \eqref{E:MP_10} reduces to
  \equ[E:MP_12]{ 0 = - \sum_{j=1}^N b_j(s,t) - \sum_{j=1}^N \E \left[ \left( X_t - X_s - \sum_{k=1}^N b_k(s,t) J^k_s \right) I^j_T\right]. }
  
  By~\eqref{E:MP_8} and~\eqref{E:MP_11}, for all $0 \leq u \leq s$,
  \al{ \E \left[Z^*_{s,t} X^{(A)}_u\right] &= \E \left[Z^*_{s,t} X_u\right] - \sum_{j=1}^N \E \left[X_u I^j_T \right] \E \left[ Z^*_{s,t} I^j_T \right] = \E \left[Z^*_{s,t} X_u\right]. }
  We replace $Z^*_{s,t}$ by~\eqref{E:MP_10b} and obtain
  \al{ \E \left[Z^*_{s,t} X^{(A)}_u\right] &= \E \left[ (X_t - X_s) X_u \right] - \sum_{i=1}^N b_i(s,t) \E \left[ J^i_s X_u \right] \\
         &\qquad \qquad - \sum_{i=1}^N \E \left[ \left( X_t - X_s - \sum_{j=1}^N b_j(s,t) J^j_s \right) I^i_T \right] \E \left[ I^i_T X_u \right]. }
  Since, $\E \left[ (X_t - X_s) X_u \right] = 0$ and $\E \left[ I^i_T X_u \right] = \E \left[ J^i_s X_u \right]$ by~\eqref{E:MP_5} and~\eqref{E:MP_6} it follows 
  \al{ \E \left[Z^*_{s,t} X^{(A)}_u\right] 
       = \E \left[ J^i_s X_u \right] \left( - \sum_{i=1}^N b_i(s,t) - \sum_{i=1}^N \E \left[ \left( X_t - X_s - \sum_{j=1}^N b_j(s,t) J^j_s \right) I^i_T \right] \right) }
  and thus $\E \left[ Z^*_{s,t} X^{(A)}_u \right] = 0$ by~\eqref{E:MP_12}. This implies $Z^*_{s,t} = Z_{s,t}$. Hence, the theorem is proven for the case that $X$ is standard linear Brownian motion.
\end{proof}

We now turn to the general case. In~\cite{Bor82} it was shown that, for Gaussian Markov processes $X=(X_s)_{s \in [0,T]}$ with $R_X(s,t) \neq 0$ for all $0 < s,t < T$, there are (up to a constant) uniquely defined functions $f: [0,T] \rightarrow \R$ and $g: [0,T] \rightarrow \R$ such that $h= f/g$ (with the convention $0/0=0$) is a non-negative, non-decreasing function on $[0,T]$ and
\[ R_X(s,t) = f(s \wedge t) g(s \vee t), \quad 0 \leq s,t \leq T. \]
This implies
\[ X_s = g(s) W_{h(s)} \]
in finite-dimensional distributions, where $W = (W_s)_{s \geq 0}$ is a standard linear Brownian motion:
\al{ \E g(s) W_{h(s)} g(t) W_{h(t)} &= g(s) g(t) ( h(s) \wedge h(t) ) = g(s) g(t)  h(s \wedge t) \\
       &= g(s) g(t)  \frac{f(s \wedge t)}{g(s \wedge t)} = f(s \wedge t) g(s \vee t). }

\begin{proof}[Proof of Theorem~\ref{T:MP} in the general case]
  We proceed in two steps: (i) we show Theorem~\ref{T:MP} for $(\tilde{g}(s) W_s)_{s \in [0,T]}$ for every positive function $\tilde{g}$; (ii) we prove the theorem for the process $(\tilde{X}_{h(s)})_{s \in [0,T]}$, where we assume the correctness of the theorem for the process $\tilde{X}$.

  Then, let $h^{-1}$ be the inverse function of $h$ (which exists since $h$ is a non-decreasing function), i.e., we have $h^{-1}(h(s)) = s$ for all $0 \leq s \leq T$, and define $\tilde{g} = g \circ h^{-1}$. By~(i), Theorem~\ref{T:MP} holds true for $\tilde{X} = \tilde{g} W$ and thus, by~(ii), Theorem~\ref{T:MP} holds true for $X = \tilde{X} \circ h$, i.e.,
  \[ X_s = \tilde{X}_{h(s)} = \tilde{g}(h(s)) W_{h(s)} = (g \circ h^{-1} \circ h)(s) W_{h(s)} = g(s) W_{h(s)}. \]

  We prove~(i): the Brownian motion $W$ and the process $\tilde{X} = \tilde{g} W$ on $[0,T]$ are generated by $u:L_2([0,T]) \rightarrow C([0,T])$ and $u_{\tilde{g}}:L_2([0,T]) \rightarrow C([0,T])$ with
  \[ (ue)(s) = \int_0^s e(x) \, dx, \qquad (u_{\tilde{g}} e)(s) = \tilde{g}(s) \int_0^s e(x) \, dx, \]
  $e \in L_2([0,T])$, $0 \leq s \leq T$. Define measures $a^{\tilde{g}}_i$ by $a^{\tilde{g}}_i(B) = \int_B \tilde{g}(x) \, a_i(dx)$, $B \in \B([0,T])$, $1 \leq i \leq N$. Then, $a_i(\tilde{X}) = a^{\tilde{g}}_i(W)$ and $(u^*a^{\tilde{g}}_i)(x) = (u_{\tilde{g}}^*a_i)(x) = a^{\tilde{g}}_i(\Ind_{[x,T]})$. Hence,
  \[ (u_{\tilde{g}} u_{\tilde{g}}^* a_i)(s) = \tilde{g}(s) \int_0^s a^{\tilde{g}}_i(\Ind_{[x,T]}) \, dx = \tilde{g}(s) (u u^* a^{\tilde{g}}_i)(s). \]
  By Proposition~\ref{P:ON_COND} we may assume that the random variables $a_1(\tilde{X}), \ldots, a_N(\tilde{X})$ are independent standard normal and thus, for $W^{(A^{\tilde{g}})}$ being the conditioned process of $W$ by $A^{\tilde{g}} = \{ a^{\tilde{g}}_1, \ldots, a^{\tilde{g}}_N \}$ and $\tilde{X}^{(A)}$ being the conditioned process of $\tilde{X}$ by $A$, $0 \leq s, t \leq T$,
  \al{ \E \left[ \tilde{X}^{(A)}_s \tilde{X}^{(A)}_t \right] &= \tilde{g}(s) \tilde{g}(t) (s \wedge t) - \sum_{i=1}^N (u_{\tilde{g}} u_{\tilde{g}}^* a_i)(s) (u_{\tilde{g}} u_{\tilde{g}}^* a_i)(t) \\
         &= \tilde{g}(s) \tilde{g}(t) \left( s \wedge t - \sum_{i=1}^N (u u^* a^{\tilde{g}}_i)(s) (u u^* a^{\tilde{g}}_i)(t) \right) \\
         &= \E \left[ \tilde{g}(s) W^{(A^{\tilde{g}})}_s \tilde{g}(t) W^{(A^{\tilde{g}})}_t \right], }
  i.e., the processes $\tilde{X}^{(A)}$ and $\tilde{g} W^{(A^{\tilde{g}})}$ coincide in law. Consider the integrated processes $I^{(A), i}$ and $L^{(A^{\tilde{g}}), i}$ given by
  \[ I^{(A), i}_s = \int_0^s \tilde{X}^{(A)}_x \, a_i(dx), \quad L^{(A^{\tilde{g}}), i}_s = \int_0^s W^{(A^{\tilde{g}})}_x \, a^{\tilde{g}}_i(dx), \quad 1 \leq i \leq N. \]
  From the proof of Theorem~\ref{T:MP} for the case that $X$ is Brownian motion we know that $(W^{(A^{\tilde{g}})}, L^{(A^{\tilde{g}}), 1}, \ldots, L^{(A^{\tilde{g}}), N})$ is a Markov process. Since $\tilde{X}^{(A)}/\tilde{g}$ and $W^{(A^{\tilde{g}})}$ coincide in law this implies that $(\tilde{X}^{(A)}/\tilde{g}, I^{(A),1}, \ldots, I^{(A),N})$ is a Markov process, where we used
  \[ \int_0^s \tilde{X}^{(A)}_x / \tilde{g}(x) \, a^{\tilde{g}}_i(dx) = \int_0^s \tilde{X}^{(A)}_x \, a_i(dx) = I^{(A),i}_s, \quad 1 \leq i \leq N. \]
  Finally, this implies that $(\tilde{X}^{(a)}, I^{(A),1}, \ldots, I^{(A),N})$ is a Markov process, which proves~(i).

  We prove~(ii): Assume that Theorem~\ref{T:MP} holds true for $\tilde{X} = (\tilde{X}_s)_{s \in [0, T']}$ and let $\tilde{X}$ be generated by $u: H \rightarrow C([0,T'])$. Moreover, let $h$ be a non-negative, increasing function on $[0,T]$ with $h(T) = T'$. Define $X = (X_s)_{s \in [0,T]} = (\tilde{X}_{h(s)})_{s \in [0,T]}$. Then $X$ is generated by $u_h: H \rightarrow C([0,T])$ with $(u_h e)(s) = (ue)(h(s))$, $e \in H$. Define measures $a^h_i$ by $a^h_i(B) = (a_i \circ h^{-1})(B)$, $B \in \B([0,T])$, $1 \leq i \leq N$. Then,
  \equ[E:MP_13]{ a_i(X) = \int_0^T X_x \, a_i(dx) = \int_0^T \tilde{X}_{h(x)} \, a_i(dx) = \int_{h(0)}^{h(T)} \tilde{X}_x \, (a \circ h^{-1})(dx). }
  If $h(0) > 0$ then, since $h$ is increasing, $h^{-1}([0, h(0))) = \emptyset$, and thus
  \equ[E:MP_14]{ a_i(X) = \int_0^{T'} \tilde{X}_x \, (a \circ h^{-1})(dx) = a^h_i(\tilde{X}). }
  In the same way we get for all $e \in H$,
  \[ \la u_h^*a_i, e \ra = \int_0^T (ue)(h(x)) \, a_i(dx) = \int_0^{T'} (ue)(x) \, a^h_i(dx) = \la u^*a^h_i, e \ra \]
  and thus $u_h^*a_i = u^*a^h_i$, $1 \leq i \leq N$. By Proposition~\ref{P:ON_COND} we may assume that the random variables $a_1(X), \ldots, a_N(X)$ are independent standard normal and thus, for $X^{(A)}$ being the conditioned process of $X$ with respect to $A$ and $\tilde{X}^{(A^h)}$ being the conditioned process of $\tilde{X}$ with respect to $A^h = \{ a^h_1, \ldots, a^h_N \}$, $0 \leq s, t \leq T$,
  \al{ \E \left[ X^{(A)}_s X^{(A)}_t \right] &= \E \left[ X_s X_t \right] - \sum_{i=1}^N (u_h u_h^* a_i)(s) (u_h u_h^* a_i)(t) \\
         &= \E \left[ \tilde{X}_{h(s)} \tilde{X}_{h(t)} \right] - \sum_{i=1}^N (u u^* a^h_i)(h(s)) (u u^* a^h_i)(h(t))  \\
         &= \E \left[ \tilde{X}^{(A^h)}_{h(s)} \tilde{X}^{(A^h)}_{h(t)} \right], }
  i.e., the processes $X^{(A)}_\cdot$ and $\tilde{X}^{(A^h)}_{h(\cdot)}$ coincide in law. Consider the integrated processes $L^{(A^h), i}$ given by $L^{(A^h), i}_s = \int_0^s \tilde{X}^{(A^h)}_x \, a^h_i(dx)$, $1 \leq i \leq N$. Then, as in~\eqref{E:MP_13} and~\eqref{E:MP_14}, for $0 \leq s \leq T$,
  \[ I_s^{(A),i} = \int_0^s X_x^{(A)} \, a_i(dx) = \int_0^s \tilde{X}_{h(x)}^{(A^h)} \, a_i(dx) = \int_0^{h(s)} \tilde{X}^{(A^h)}_x \, a^h_i(dx) = L^{(A^h)}_{h(s)} \]
  in finite-dimensional distributions. By the assumption on $\tilde{X}$ the process \linebreak $(\tilde{X}^{(A^h)}_s, L^{(A^h), 1}_s, \ldots, L^{(A^h), N}_s)_{s \in [0,T']}$ is a Markov process implying that \linebreak $(\tilde{X}^{(A^h)}_{h(s)}, L^{(A^h), 1}_{h(s)}, \ldots, L^{(A^h), N}_{h(s)})_{s \in [0,T]}$ is a Markov process as well. Since \linebreak $(X^{(A)}_s, I^{(A),1}_s, \ldots, I^{(A),N}_s)_{s \in [0,T]}$ and $(\tilde{X}^{(A^h)}_{h(s)}, L^{(A^h), 1}_{h(s)}, \ldots, L^{(A^h), N}_{h(s)})_{s \in [0,T]}$ coincide in law we conclude that $(X^{(A)}_s, I^{(A),1}_s, \ldots, I^{(A),N}_s)_{s \in [0,T]}$ is a Markov process.
\end{proof}

\subsection{The expected future}

Now, we can give an explicit formula for $\E[X^{(A)}_t | \F^{X^{(A)}}_s]$, $s < t \leq T$. This together with Theorem~\ref{T:DRIFT} enables us to calculate the drift term in Theorem~\ref{T:SDE_EX} in the case that $X$ is Markovian. Define a matrix $D_s$ by
\[ D_s = \left( \begin{matrix} g(s) & (ue_1)(s) & \dots & (ue_N)(s) \\ \int_{s+}^T g(x) \, a_1(dx) & \int_{s+}^T (ue_1)(x) \, a_1(dx) & \dots & \int_{s+}^T (ue_N)(x) \, a_1(dx) \\ \vdots & \vdots & \ddots & \vdots \\ \int_{s+}^T g(x) \, a_N(dx) & \int_{s+}^T (ue_1)(x) \, a_N(dx) & \dots & \int_{s+}^T (ue_N)(x) \, a_N(dx) \end{matrix} \right) \]
and a vector $d_s$ by
\[ d_s = \left( X^{(A)}_s, - I^{(A), 1}_s, \ldots, - I^{(A), N}_s \right)^\tau. \]

\begin{theorem}\label{T:EXP_FUTURE}
  For every $s<t$ there are $\F^{X^{(A)}}_s$-measurable random variables $\xi_0, \ldots, \xi_N$ such that
  \[ \E[X^{(A)}_t | \F^{X^{(A)}}_s] = \xi_0 g(t) + \sum_{i=1}^N \xi_i (ue_i)(t). \]
  Assume that the matrix $D_s$ is invertible. Then $\xi = (\xi_0, \ldots, \xi_N)^\tau$ is given by $\xi = D_s^{-1} d_s$.
\end{theorem}

\begin{proof}
  For $s < t$ and $1 \leq i \leq N$, we have
  \al{ \E \left[ I^{(A), i}_s X^{(A)}_t \right] &= \E \left[ \int_0^s X^{(A)}_x \, a_i(dx) X^{(A)}_t \right] 
         = \int_0^s \E \left[ X^{(A)}_x X^{(A)}_t \right] \, a_i(dx) \\
         &= \int_0^s R_{X^{(A)}}(x,t) \, a_i(dx) \\
         &= \int_0^s \left( f(x)g(t) - \sum_{i=1}^N (ue_i)(x)(ue_i)(t) \right) \, a_i(dx) \\
         &= g(t) \int_0^s f(x) \, a_i(dx) - \sum_{i=1}^N (ue_i)(t) \int_0^s (ue_i)(x) \, a_i(dx). }
  In particular, $\E \left[ I^{(A), i}_s X^{(A)}_t \right]$ is a deterministic linear combination of $g(t)$ and $(ue_i)(t)$, $1 \leq i \leq N$.

  By Theorem~\ref{T:MP},
  \[ \E[X^{(A)}_t | \F^{X^{(A)}}_s] = \E[ X^{(A)}_t \mid \{ X^{(A)}_s, I^{(A), 1}_s, \ldots, I^{(A), N}_s \} ]. \]
  Assume without loss of generality that $\{ X^{(A)}_s, I^{(A), 1}_s, \ldots, I^{(A), N}_s \}$ are orthonormal random variables (otherwise orthonormalize them similar to~\eqref{E:ANTI_1}). Then, by the general theory of conditioning of Gaussian random variables,
  \[ \E[X^{(A)}_t | \F^{X^{(A)}}_s] = X^{(A)}_s \E \left[ X^{(A)}_s X^{(A)}_t \right] + \sum_{i=1}^N I^{(A), i}_s \E \left[ X^{(A)}_t I^{(A), i}_s \right]. \]
  Since $\E \left[ X^{(A)}_s X^{(A)}_t \right]$ and $\E \left[ I^{(A), i}_s X^{(A)}_t \right]$ are deterministic linear combinations of $g(t)$ and $(ue_i)(t)$, $1 \leq i \leq N$, there are $\F^{X^{(A)}}_s$-measurable random variables $\xi_0, \ldots, \xi_N$ such that
  \[ \E[X^{(A)}_t | \F^{X^{(A)}}_s] = \xi_0 g(t) + \sum_{i=1}^N \xi_i (ue_i)(t). \]

  In order to determine $\xi_0, \ldots, \xi_N$, consider the process $Z = (Z_t)_{t \in [0,T]}$ defined by
  \[ Z_t = \begin{cases} X^{(A)}_t, &\text{for $t \leq s$,} \\ \E[X^{(A)}_t | \F^{X^{(A)}}_s], &\text{for $t > s$.} \end{cases} \]
  $Z$ is continuous and fulfills the conditions $a_1, \ldots, a_N$, i.e.,
  \[ Z_s = X^{(A)}_s = \lim_{t \searrow s} Z_t = \lim_{t \searrow s} \E[X^{(A)}_t | \F^{X^{(A)}}_s] = \xi_0 g(s) + \sum_{i=1}^N \xi_i (ue_i)(s), \]
  and
  \al{ 0 = \int_0^T Z_x \, a_j(dx) &= \int_0^s X^{(A)}_x \, a_j(dx) + \int_{s+}^T \E[X^{(A)}_x | \F^{X^{(A)}}_s] \, a_j(dx) \\
             &= I^{(A), j}_s + \xi_0 \int_{s+}^T g(x) \, a_j(dx) + \sum_{i=1}^N \xi_i \int_{s+}^T (ue_i)(x) \, a_j(dx), }
  i.e.,
  \[ - I^{(A), j}_s = \xi_0 \int_{s+}^T g(x) \, a_j(dx) + \sum_{i=1}^N \xi_i \int_{s+}^T (ue_i)(x) \, a_j(dx), \quad 1 \leq j \leq N. \]
  This leads to the system of linear equations $D_s \xi = d_s$ and its solution $\xi = D_s^{-1} d_s$.
\end{proof}


\section{Examples}\label{S:EXAMPLES}

\subsection{The zero area Brownian bridge}\label{SS:ZABB}

The standard linear Brownian motion $W=(W_s)_{s \in [0,1]}$ on $[0,1]$ is generated by the operator $u: L_2([0,1]) \rightarrow C([0,1])$ with
\[ (uh)(s) = \int_0^s h(x) \, dx \]
for $h \in L_2([0,1])$. For example, the trigonometric basis in $L_2([0,1])$,
\[ \{ e_n : n \geq 0 \} = \{ 1 \} \cup \{ \sqrt{2} \cos(\pi n x) : n \geq 1 \}, \]
for which $(u e_0)(s) = s$ and
\[ (u e_n)(s) = \int_0^s \sqrt{2} \cos(\pi n x) \, dx = \sqrt{2} \frac{\sin(\pi n s)}{\pi n}, \]
yields the well known representation
\[ W_s = \omega_0 s + \sqrt{2} \sum_{n=1}^\infty \omega_n \frac{\sin(\pi n s)}{\pi n}. \]

Let $M=(M_s)_{s \in [0,1]}$ be the Brownian motion conditioned to be zero at time $1$ and with integral zero, i.e., $M = W^{(A)}$ for $A = \{ \delta_1, a_0 \} \subset C([0,1])^*$ with
\[ \delta_1(f) = f(1) \qquad \text{and} \qquad a_0(f) = \int_0^1 f(s) \, ds, \qquad f \in C([0,1]). \]
It holds
\[ (u^*\delta_1)(x) = 1 \qquad \text{and} \qquad (u^*a_0)(x) = 1-x. \]
The detached subspace $H_{(A)}$ of $L_2([0,1])$ with respect to the set of conditions $A = \{ \delta_1, a_0 \} \subset C([0,1])^*$ is thus $H_{(A)} = \setspan \{ 1, 1-x \}$. An orthonormal basis in $H_{(A)}$ is $\{e_1, e_2\} = \{ 1, \sqrt{3}(1-2x) \}$. Hence, according to Proposition~\ref{P:COND_COVARIANCE}, the covariance of the zero area Brownian bridge $M=W^{(A)}$ is given by ($0 \leq s,t \leq 1$)
\al{ R_M(s,t) &= R_W(s, t) - (ue_1)(s)(ue_1)(t) - (ue_2)(s)(ue_2)(s) \\
              &= \min\{s,t\} - \int_0^s \, dx \int_0^t \, dy - \int_0^s \sqrt{3}(1-2x) \, dx \int_0^t \sqrt{3}(1-2y) \, dy \\
              &= \min\{s,t\} - st - 3(s- s^2)(t-t^2). }

Using the notation from Theorem~\ref{T:ANTI} the matrix $B$ and the vector $b$ become
\[ B = \left( \begin{matrix} \delta_1(ue_1) & \delta_1(ue_2) \\ a_0(ue_1) & a_0(ue_2) \end{matrix} \right) = \left( \begin{matrix} 1 & 0 \\ 1/2 & 1/(2 \sqrt{3}) \end{matrix} \right) \quad \text{and} \quad b = \left( \begin{matrix} \delta_1(W) \\ a_0(W) \end{matrix} \right) = \left( \begin{matrix} W_1 \\ I_1 \end{matrix} \right), \]
where $I_s = \int_0^s W_x \, dx$. Solving the linear equation system $B \xi = b$ yields
\[ \xi_1 = W_1 \qquad \text{and} \qquad \xi_2 = \sqrt{3}(2 I_1 - W_1). \]
Then, by Theorem~\ref{T:ANTI}, an anticipative representation for $M$ is
\al{ M_s &= W_s - W_1 s - \sqrt{3}(2 I_1 - W_1) \sqrt{3}(s-s^2) \\
         &= W_s - s(3s - 2) W_1 - 6 s (1-s) I_1. }

Let $\Prob_W$ and $\Prob_M$ be the induced measures of $W$ and $M$ on $(C([0,1]), \mathcal{C})$. For every $s < 1$ the condition in~\eqref{E:EQU_MEAS_1} is fulfilled. Hence, by Theorem~\ref{T:EQU_MEASURES}, the measures $\Prob_W$ and $\Prob_M$ are equivalent on $\F_s$ for every $s < 1$, where, as in Theorem~\ref{T:EQU_MEASURES}, $\F_s \subset \mathcal{C}$ is the smallest $\sigma$-algebra on $C([0,T])$ such that all point evaluation functionals $\delta_x$, $0 \leq x \leq s$, are $\F_s$-$\B(\R)$-measurable.

By Theorem~\ref{T:SDE_EX}, $M$ is a solution of the stochastic differential equation
\[ dM_s = dW_s + \delta(s, M) ds, \quad M_0 = 0, \quad 0 \leq s < 1, \]
where $\delta$ is a progressively measurable functional on $C([0,1])$. By Theorem~\ref{T:DRIFT},
\[ \delta(s, M) = \lim_{r \searrow 0} \frac{\E[ M_{s+r} \mid \F_s^M] - M_s}{r}, \quad 0 \leq s < 1, \]
where $\F_s^M$ is the natural filtration of $M$ at time $s$. Define $J_s = \int_0^s M_x \, dx$, $0 \leq s \leq 1$. Since $(W_s)_{s \in [0,1]}$ is a Markov process, $(M_s, J_s)_{s \in [0,1]}$ is a Markov process as well by Theorem~\ref{T:MP}. By Theorem~\ref{T:EXP_FUTURE}, for $0 \leq s \leq t < 1$, we have
\[ \E[ M_t \mid \F_s^M] = \xi_0 + \xi_1 t + \xi_2 \sqrt{3}(t-t^2), \]
where $\xi = (\xi_0, \xi_1, \xi_2)$ is the solution of the system of linear equations $D_s \xi = d_s$ with $d_s = (M_s, 0, -J_s)$ and
\[ D_s = \left( \begin{matrix} 1 & s & \sqrt{3}(s-s^2) \\ 1 & 1 & 0 \\ 1-s & (1-s^2)/2 & \sqrt{3}(1-s^2)/2 - (1 - s^3)/\sqrt{3} \end{matrix} \right). \]
Solving this system of linear equations yields
\al{ \xi_0 &= \frac{M_s (2s^2 - s - 1) - 6 J_s s}{(s-1)^3}, \\
     \xi_1 &= - \frac{M_s (2s^2 - s - 1) - 6 J_s s}{(s-1)^3}, \\
     \xi_2 &= - \sqrt{3} \frac{M_s(s-1) - 2 J_s}{(s-1)^3}, }
and thus
\al{ \E[ M_t \mid \F_s^M] = \frac{M_s (2s^2 - s - 1) - 6 J_s s}{(s-1)^3} - t \ \frac{M_s (2s^2 - s - 1) - 6 J_s s}{(s-1)^3} - 3(t-t^2) \ \frac{M_s(s-1) - 2 J_s}{(s-1)^3}. }
We have
\[ \lim_{r \searrow 0} \frac{\E[ M_{s+r} \mid \F_s^M] - M_s}{r} = - \frac{4M_s}{1-s} - \frac{6 J_s}{(1-s)^2}. \]
Hence, $M$ has the stochastic differential
\[ dM_s = dW_s - \frac{4M_s}{1-s} ds - \frac{6 J_s}{(1-s)^2} ds, \quad M_0 = 0, \quad 0 \leq s < 1. \]

\subsection{Gaussian bridges}\label{SS:BRIDGES}

The conditioning of a Gaussian process on $[0,T]$ to be zero at time $T$ is a well-studied but important example (see for example~\cite{Gas04}). This leads to Gaussian bridges: let $X = (X_s)_{s \in [0,T]}$ be a continuous Gaussian process and let $\delta_T \in C([0,T])^*$ be the evaluation functional at point $T$. Then $X^{(\delta_T)}$ is called the bridge process of $X$.

\begin{prop}
  The covariance $R_{X^{(\delta_T)}}(s,t) = \E X^{(\delta_T)}_s X^{(\delta_T)}_t$ is
  \[ R_{X^{(\delta_T)}}(s,t) = R_X(s,t) - \frac{R_X(s,T)R_X(t,T)}{R_X(T,T)}, \qquad 0 \leq s,t \leq T, \]
  where $R_X(s,t) = \E X_s X_t$ is the covariance function of $X$, and a anticipative representation for $X^{(\delta_T)}$ is
  \[ X^{(\delta_T)}_s = X_s - \frac{R_X(s,T)}{R_X(T,T)} X_T, \quad 0 \leq s \leq T. \]
\end{prop}

\begin{proof}
  Let $X$ be generated by the linear and bounded operator $u:H \rightarrow C([0,T])$ and let $(e_i)_{i=1}^\infty$ be an orthonormal basis in the separable Hilbert space $H$. By~\eqref{E:CONDITIONS}, the detached Hilbert space $H_{(\delta_T)}$ with respect to the condition $\delta_T$ is spanned by
  \[ u^*\delta_T = \sum_{i=1}^\infty \la u^*\delta_T, e_i \ra e_i = \sum_{i=1}^\infty (u e_i)(T) e_i. \]
  By Parseval's identity and~\eqref{E:CP_4},
  \[ \|u^*\delta_T\|^2 = \sum_{i=1}^\infty |\la u^*\delta_T, e_i \ra|^2 = \sum_{i=1}^\infty (ue_i)(T) (ue_i)(T) = R_X(T,T). \]
  Hence, by Proposition~\ref{P:COND_COVARIANCE} in the first line and~\eqref{E:CP_4} in the second line
  \al{ R_{X^{(\delta_T)}}(s,t) &= R_X(s, t) - \frac{\sum_{i=1}^\infty (ue_i)(T)(ue_i)(s) \sum_{i=1}^\infty (ue_i)(T)(ue_i)(t)}{R_X(T,T)} \\
                               &= R_X(s,t) - \frac{R_X(s,T) R_X(t,T)}{R_X(T,T)}. }
  The anticipative representation of $X^{(\delta_T)}$ follows by Theorem~\ref{T:ANTI}.
\end{proof}


\section*{Acknowledgments}

The author would like to thank Ingemar Kaj and Svante Janson for valuable comments.


\bibliographystyle{plain}
\bibliography{bibl}


\end{document}